\documentclass{amsart}

\usepackage{graphicx}

\newtheorem{theorem}{Theorem}[section]
\newtheorem{lemma}[theorem]{Lemma}

\theoremstyle{definition}

\newtheorem{corollary}{Corollary}[section]

\theoremstyle{remark}
\newtheorem{remark}[theorem]{Remark}

\numberwithin{equation}{section}

\begin{document}

\title{Improved estimate for the prime counting function}

\author{Theophilus Agama}
\address{Department of Mathematics, African Institute for Mathematical science, Ghana
}
\email{theophilus@aims.edu.gh/emperordagama@yahoo.com}


\subjclass[2000]{Primary 54C40, 14E20; Secondary 46E25, 20C20}

\date{\today}

\dedicatory{}

\keywords{prime, theta}

\begin{abstract}
 Using some simple combinatorial arguments, we establish some new estimates for the prime counting function and its allied functions. In particular we show that \begin{align}\pi(x)=\Theta(x)+O\bigg(\frac{1}{\log x}\bigg), \nonumber
\end{align}where \begin{align}\Theta(x)=\frac{\theta(x)}{\log x}+\frac{x}{2\log x}-\frac{1}{4}-\frac{\log 2}{\log x}\sum \limits_{\substack{n\leq x\\\Omega(n)=k\\k\geq 2\\2\not| n}} \frac{\log (\frac{x}{n})}{\log 2}.\nonumber
\end{align}This is an improvement to the estimate \begin{align}\pi(x)=\frac{\theta(x)}{\log x}+O\bigg(\frac{x}{\log^2 x}\bigg)\nonumber
\end{align}found in the literature.
\end{abstract}

\maketitle

\section{Introduction}
Let us set $\pi(x)=\sum \limits_{p\leq x}1$, where $\pi(x)$ counts the number of primes no more than a fixed real number. The  irregularity of the primes makes it difficult to obtain an exact pratical formula for the prime counting function. It is one of the well studied functions in number theory and the whole of mathematics given its profound connection to the grand Riemann hypothesis, but not much is known concerning their distribution. In the following sequel, our mission is to improve a known estimate for the prime counting function involving the chebyshev theta  function, given  by \begin{align}\pi(x)=\frac{\theta(x)}{\log x}+O\bigg(\frac{x}{\log^2 x}\bigg). \nonumber
\end{align}

\section{Notations}
Through out this paper a prime number will either be denoted by $p$ or the subscripts of $p$. Any other letter will be clarified. The function $\Omega(n):=\sum \limits_{p||n}1$ counts the number of prime factors of $n$ with multiplicity. The inequality $|k(n)|\leq Mp(n)$ for sufficiently large values of $n$ will be compactly written as $k(n)\ll p(n)$ or $k(n)=O(p(n))$.

\section{Preliminary results}
\begin{lemma}\label{hadarman}
There exist some constant $c>0$ such that \begin{align}\theta(x)=x+O\bigg(\frac{x}{e^{c\sqrt{\log x}}}\bigg).\nonumber
\end{align}
\end{lemma}

\begin{proof}
For a proof, see for instance \cite{May}.
\end{proof}

\begin{lemma}
For all $x\geq 2$ \begin{align}\pi(x)=\frac{\theta(x)}{\log x}+\int \limits_{2}^{x}\frac{\theta(t)}{t\log^2 t}dt.\nonumber
\end{align}
\end{lemma}

\begin{proof}
For a proof see for instance \cite{rosser1962approximate}.
\end{proof}

\begin{lemma}
For all $x\geq 2$ \begin{align}\theta(x)=\pi(x)\log x-\int \limits_{2}^{x}\frac{\pi(t)}{t}dt. \nonumber
\end{align}
\end{lemma}

\begin{proof}
For a proof see for instance \cite{rosser1962approximate}.
\end{proof}

\section{Main result}
\begin{theorem}\label{general}
For every positive integer $x$\begin{align}\sum \limits_{\substack{n\leq x\\(2,n)=1}}\left \lfloor \frac{\log (\frac{x}{n})}{\log 2}\right \rfloor=\frac{x-1}{2}+\bigg(1+(-1)^x\bigg)\frac{1}{4}. \nonumber \end{align}
\end{theorem}

\begin{proof}
The plan of attack is to examine the distribution of odd natural numbers and even numbers. We break the proof of this result into two cases; The case $x$ is odd and the case $x$ is even. For the case $x$ is odd, we argue as follows:  We first observe that there are as many even numbers as odd numbers less than any given odd number $x$. That is, for $1\leq m < x$, there are $(x-1)/2$ such possibilities. On the other hand, consider the sequence of even numbers less than $x$, given as $2, 2^2, \ldots, 2^b$ such that $2^b < x$; clearly there are $\lfloor \frac{\log x}{\log 2}\rfloor$ such number of terms in the sequence. Again consider those of the form $3\cdot 2, \ldots, 3\cdot 2^b$ such that $3\cdot 2^b <x$; Clearly there are $\left \lfloor \frac{\log (x/3)}{\log 2}\right \rfloor$ such terms in this sequence. We terminate the process by considering those of the form $2\cdot j, \ldots, 2^b\cdot j$ such that $(2,j)=1$ and $2^b\cdot j<x$; Clearly there are $\left \lfloor \frac{\log (x/j)}{\log 2}\right \rfloor$ such number of terms in this sequence. The upshot is that $(x-1)/2=\sum \limits_{\substack{j\leq x\\(2,j)=1}}\left \lfloor \frac{\log (x/j)}{\log 2}\right \rfloor$. We now turn to the case $x$ is even. For the case $x$ is even, we argue as follows:  First we observe that there are $x/2$ even numbers less than or equal to $x$. On the other hand, there are $\sum \limits_{\substack{j\leq x\\(2,j)=1}}\left \lfloor \frac{\log (x/j)}{\log 2}\right \rfloor$ even numbers less than or equal to $x$. This culminates into the assertion that $(x-1)/2+\frac{1}{2}=\sum \limits_{\substack{j\leq x\\(2,j)=1}}\left \lfloor \frac{\log (x/j)}{\log 2}\right \rfloor$. By combining both cases, the result follows immediately.
\end{proof}
\bigskip

\begin{remark}
Understanding the fractional parts of real numbers is an important problem in number theory and gains much paramountcy. The next result gives a little insight into the problem.
\end{remark}

\begin{corollary}\label{fractional part1}
\begin{align}\sum \limits_{n\leq x}\left \{\frac{\log \bigg(\frac{x}{n}\bigg)}{\log 2}\right \}=\frac{x}{\log 2}-x-\frac{\log x}{\log 4}+O(1).\nonumber
\end{align}
\end{corollary}

\begin{proof}
Stirling's formula \cite{robbins1955remark} gives \begin{align}\sum \limits_{n\leq x}\log n=x\log x-x+\frac{1}{2}\log x+\log(\sqrt{2\pi})+O\bigg(\frac{1}{x}\bigg).\label{son}
\end{align} Also from Theorem \ref{general}, we obtain \begin{align}\sum \limits_{n\leq x}\log n & = x\log x-x\log 2-\log 2\sum \limits_{n\leq x}\left \{\frac{\log \bigg(\frac{x}{n}\bigg)}{\log 2}\right \}+O(1).\label{com}
\end{align}Comparing equation \eqref{son} and \eqref{com}, we have \begin{align}\sum \limits_{n\leq x}\left \{\frac{\log \bigg(\frac{x}{n}\bigg)}{\log 2}\right \}=\frac{x}{\log 2}-x-\frac{\log x}{\log 4}+O(1),\nonumber
\end{align}thereby establishing the estimate.
\end{proof}
\bigskip

The prime counting function is one of the most important functions in number theory, given it's connection with the famous Riemann hypothesis. The prime counting function has been studied by many authors in the past decade. For instance in 2010, Dusart \cite{dusart2010estimates} showed that \begin{align}\pi(x)\leq \frac{x}{\log x}+\frac{x}{\log^2 x}+\frac{2.334x}{\log^3 x}\nonumber
\end{align}holds for every $x\geq 2953652287$ and that \begin{align}\pi(x)\geq \frac{x}{\log x}+\frac{x}{\log^2 x}+\frac{2x}{\log^3 x}\nonumber
\end{align}for every $x\geq 88783$. We now obtain, as an immediate consequence of Theorem \ref{general}, an explicit formula for the prime counting function $\pi(x)$. It relates the prime counting function to the Chebychev theta function $\theta(x)$.
\bigskip

\begin{theorem}\label{genuis}
For all positive integers $x$ \begin{align}\pi(x)=\frac{(x-1)\log (\sqrt{2})+\theta(x)+(\log 2)\bigg(H(x)-G(x)+T(x)\bigg) +\bigg(1+(-1)^x\bigg)\frac{\log 2}{4}}{\log x},\nonumber
\end{align}where \begin{align}H(x):= \sum \limits_{p\leq x}\left \{ \frac{\log (\frac{x}{p})}{\log 2}\right \}, \quad G(x):=\left \lfloor \frac{\log x}{\log 2}\right \rfloor +\sum \limits_{\substack{n\leq x\\\Omega(n)=k\\k\geq 2\\ 2\not|  n}}\left \lfloor \frac{\log (\frac{x}{n})}{\log 2}\right \rfloor, \quad T(x):= \left \lfloor \frac{\log (\frac{x}{2})}{\log 2}\right \rfloor \nonumber
\end{align} \begin{align}\theta(x):=\sum \limits_{p \leq x}\log p, \quad \text{and}\quad \{\cdot \} \nonumber \end{align}denotes the fractional part of any real number.
\end{theorem}

\begin{proof}
We use Theorem \ref{general} to establish an explicit formula for the prime counting function. Theorem \ref{general} gives $\displaystyle{(x-1)/2=\sum \limits_{\substack{n\leq x\\(2,n)=1}}\left \lfloor \frac{\log (x/n)}{\log 2}\right \rfloor -(1+(-1)^{x})1/4}$, which can then be recast as \begin{align}(x-1)/2&=\sum \limits_{p\leq x}\left \lfloor \frac{\log (x/p)}{\log 2}\right \rfloor -\left \lfloor \frac{\log (x/2)}{\log 2}\right \rfloor +\left \lfloor \frac{\log x}{\log 2}\right \rfloor +\sum \limits_{\substack{n\leq x\\\Omega(n)=k\\k\geq 2\\2\not| n}}\left \lfloor \frac{\log (x/n)}{\log 2}\right \rfloor -\bigg(1+(-1)^x\bigg)/4, \nonumber
\end{align}where $p$ runs over the primes. It follows by futher simplification that \begin{align}(x-1)/2=\frac{1}{\log 2}\bigg(\log x\sum \limits_{p\leq x}1-\sum \limits_{p\leq x}\log p\bigg)-\sum \limits_{p\leq x}\left \{\frac{\log (x/p)}{\log 2}\right \}+\left \lfloor \frac{\log x}{\log 2}\right \rfloor -\left \lfloor \frac{\log (x/2)}{\log 2}\right \rfloor \nonumber
\end{align}\begin{align}+\sum \limits_{\substack{n\leq x\\\Omega(n)=k\\k\geq 2\\2\not| n}}\left \lfloor \frac{\log (x/n)}{\log 2}\right \rfloor -\bigg(1+(-1)^x\bigg)/4. \nonumber \end{align} It follows that\begin{align}(x-1)\log \sqrt{2}=\pi(x)\log x-\theta(x)-(\log 2)\bigg(H(x)-G(x)+T(x)\bigg)-\bigg(1+(-1)^x\bigg)/4, \nonumber
\end{align}where \begin{align}H(x):= \sum \limits_{p\leq x}\left \{ \frac{\log (\frac{x}{p})}{\log 2}\right \}, \quad G(x):=\left \lfloor \frac{\log x}{\log 2}\right \rfloor +\sum \limits_{\substack{n\leq x\\\Omega(n)=k\\k\geq 2\\2\not| n}}\left \lfloor \frac{\log (\frac{x}{n})}{\log 2}\right \rfloor, \quad T(x):= \left \lfloor \frac{\log (\frac{x}{2})}{\log 2}\right \rfloor, \nonumber
\end{align}\begin{align}\text{and} \quad \theta(x):=\sum \limits_{p \leq x}\log p, \nonumber \end{align}and the result follows immediately.
\end{proof}
\bigskip

\begin{remark}
Using Theorem \ref{genuis}, we obtain an estimate for the prime counting and it's allied functions, with very small error term in the following sequel. 
\end{remark}

\begin{theorem}\label{prime counting}
For all positive integers $x\geq 2$\begin{align}\pi(x)=\Theta(x)+O\bigg(\frac{1}{\log x}\bigg), \nonumber
\end{align}where \begin{align}\Theta(x)=\frac{\theta(x)}{\log x}+\frac{x}{2\log x}-\frac{1}{4}-\frac{\log 2}{\log x}\sum \limits_{\substack{n\leq x\\\Omega(n)=k\\k\geq 2\\2\not| n}} \frac{\log (\frac{x}{n})}{\log 2}.\nonumber
\end{align}
\end{theorem}

\begin{proof}
By Theorem \ref{genuis}, we can write\begin{align}\pi(x)=\bigg(\frac{\log 2}{2}\bigg)\frac{x}{\log x}+\frac{\theta(x)}{\log x}+\frac{\log 2}{\log x}\bigg(H(x)-G(x)+T(x)\bigg)+O\bigg(\frac{1}{\log x}\bigg),\nonumber
\end{align}where \begin{align}H(x):= \sum \limits_{p\leq x}\left \{ \frac{\log (\frac{x}{p})}{\log 2}\right \}, \quad G(x):=\left \lfloor \frac{\log x}{\log 2}\right \rfloor +\sum \limits_{\substack{n\leq x\\\Omega(n)=k\\k\geq 2\\2\not| n}}\left \lfloor \frac{\log (\frac{x}{n})}{\log 2}\right \rfloor, \quad T(x):= \left \lfloor \frac{\log (\frac{x}{2})}{\log 2}\right \rfloor. \nonumber
\end{align}Now, we estimate the term $H(x)-G(x)+T(x)$. Clearly we can write \begin{align}-G(x)+T(x)&=-\sum \limits_{\substack{n\leq x\\\Omega(n)=k\\k\geq 2\\2\not| n}}\left \lfloor \frac{\log (\frac{x}{n})}{\log 2}\right \rfloor+O(1)\nonumber
\end{align}It follows that \begin{align}H(x)-G(x)+T(x)&=\sum \limits_{\substack{n\leq x\\(2,n)=1}}\left \{\frac{\log \bigg(\frac{x}{n}\bigg)}{\log 2}\right \}-\sum \limits_{\substack{n\leq x\\\Omega(n)=k\\k\geq 2\\2\not| n}}\frac{\log (\frac{x}{n})}{\log 2}+O(1)\nonumber \\&=\frac{x}{\log 4}-\frac{x}{2}-\frac{\log x}{\log 16}-\sum \limits_{\substack{n\leq x\\\Omega(n)=k\\k\geq 2\\2\not| n}} \frac{\log (\frac{x}{n})}{\log 2}+O(1),\nonumber
\end{align}where we have used Corollary \ref{fractional part1} and the proof is complete. 
\end{proof}

\begin{corollary}
There exist a contant $c>0$ such that \begin{align}\pi(x)=\nu(x)+O\bigg(\frac{x}{e^{c\sqrt{\log x}}\log x}\bigg),\nonumber
\end{align}where \begin{align}\nu(x):=\frac{3x}{2\log x}-\frac{1}{4}-\frac{\log 2}{\log x}\sum \limits_{\substack{n\leq x\\\Omega(n)=k\\k\geq 2\\2\not| n}}\frac{\log (\frac{x}{n})}{\log 2}.\nonumber
\end{align}
\end{corollary}

\begin{proof}
The result follows by plugging the estimate in Lemma \ref{hadarman} into Theorem \ref{prime counting}.
\end{proof}
\bigskip

It is known that \begin{align}\pi(x)=\frac{\theta(x)}{\log x}+O\bigg(\frac{x}{\log^2 x}\bigg)\nonumber
\end{align}\cite{May}, where $\theta(x):=\sum \limits_{p\leq x}\log p$. The above estimate derived indicates that we can do better than this with \begin{align}\pi(x)=\Theta(x)+O\bigg(\frac{1}{\log x}\bigg),\nonumber
\end{align}where $\Theta(x)$ is some explicit function.
\bigskip

\begin{remark}
It needs to be said that Theorem \ref{prime counting} relates the prime counting function to the Chebyshev theta function. Thus, very good and sharp estimates for $\theta(x)$ will eventually yield a sharp estimate for the prime counting function. For the purpose of applications, we stick to the current estimate involving the Chebyshev theta function.
\end{remark}

\begin{theorem}\label{integral}
For all positive integers $x\geq 2$, we have \begin{align}\int \limits_{2}^{x}\frac{\theta(t)}{t\log ^2 t}dt=R(x)+O\bigg(\frac{1}{\log x}\bigg),\nonumber
\end{align}where \begin{align}R(x)=\frac{x}{2\log x}-\frac{1}{4}-\frac{\log 2}{\log x}\sum \limits_{\substack{n\leq x\\\Omega(n)=k\\k\geq 2\\2\not| n}}\frac{\log (\frac{x}{n})}{\log 2}.\nonumber
\end{align}
\end{theorem}

\begin{proof}
It is known that \begin{align}\pi(x)=\frac{\theta(x)}{\log x}+\int \limits_{2}^{x}\frac{\theta(t)}{t\log^2 t}dt\nonumber
\end{align}for all $x\geq 2$ \cite{rosser1962approximate}. Comparing this result with Theorem \ref{prime counting}, the result follows immediately.
\end{proof}
\bigskip

\begin{theorem}\label{integral 2}
For all positive integers $x\geq 2$, we have \begin{align}\int \limits_{2}^{x}\frac{\pi(t)}{t}dt=\eta(x)+O(1),\nonumber 
\end{align}where \begin{align}\eta(x)=\frac{x}{2}-\frac{\log x}{4}-\log 2\sum \limits_{\substack{n\leq x\\\Omega(n)=k\\k\geq 2\\2\not| n}}\frac{\log (\frac{x}{n})}{\log 2}.\nonumber
\end{align}
\end{theorem}

\begin{proof}
Recall the well-known identity \cite{rosser1962approximate} \begin{align}\theta(x)=\pi(x)\log x-\int \limits_{2}^{x}\frac{\pi(t)}{t}dt. \nonumber
\end{align} By comparing this identity with Theorem \ref{prime counting}, the result follows immediately.
\end{proof}
\bigskip

\begin{remark}
Theorem \ref{integral} measures the discrepancies of the prime counting function and $\theta(x)/\log x$.
\end{remark}
\section{Final remarks}
In this paper we have obtained a new estimate for the prime counting function. We have shown that indeed the estimate \begin{align}\pi(x)=\Theta(x)+O\bigg(\frac{1}{\log x}\bigg),\nonumber
\end{align}where \begin{align}\Theta(x)=\frac{\theta(x)}{\log x}+\frac{x}{2\log x}-\frac{1}{4}-\frac{\log 2}{\log x}\sum \limits_{\substack{n\leq x\\\Omega(n)=k\\k\geq 2\\2\not| n}} \frac{\log (\frac{x}{n})}{\log 2}\nonumber
\end{align}is better than \begin{align}\pi(x)=\frac{\theta(x)}{\log x}+O\bigg(\frac{x}{\log^2 x}\bigg)\nonumber
\end{align}found in the literature.

\footnote{
\par
.}%
.




\bibliographystyle{amsplain}

\end{document}